\documentclass[11pt]{article}
\usepackage{amsmath, amsfonts, amssymb}
\usepackage{graphicx, amsthm,color}
\usepackage{mathrsfs,dsfont}  

\usepackage[T1]{fontenc}
\DeclareUnicodeCharacter{0141}{\L{}}
\usepackage{comment}
\usepackage[numbers,sort&compress]{natbib}
\usepackage{hyperref}

\usepackage{tikz}
\usepackage{tikz-3dplot}
\usetikzlibrary{calc,patterns,angles,quotes}
\usetikzlibrary{babel}
\usepackage{pgfplots}
\usepgfplotslibrary{patchplots}
\usetikzlibrary{backgrounds, intersections}
\usepgfplotslibrary{fillbetween}
\usetikzlibrary{arrows,automata}
\usepackage{subcaption}

\pgfplotsset{compat=1.18}
\usepgfplotslibrary{colorbrewer}

\newtheorem{theorem}{Theorem}[section]
\newtheorem{lemma}[theorem]{Lemma}
\newtheorem{definition}[theorem]{Definition}
\newtheorem{corollary}[theorem]{Corollary}
\newtheorem{proposition}[theorem]{Proposition}

\theoremstyle{definition}
\newtheorem{remark}[theorem]{Remark}

\usepackage[utf8]{inputenc}

\newcommand{\HH}{\mathcal{H}}
\newcommand{\N}{\mathds{N}}

\newcommand{\R}{\mathds{R}}
\newcommand{\Sph}{\mathds{S}}

\newcommand{\cco}{\overline{\textup{co}}}
\newcommand{\co}{\textup{co}}
\newcommand{\dist}{\textup{dist}}

\DeclareMathOperator{\argmin}{argmin}

\DeclareMathOperator{\dom}{dom}
\DeclareMathOperator{\Id}{Id}

\newcommand{\tto}{\rightrightarrows}

\newcommand{\proj}{\mathrm{proj}}

\date{\today}
\title{Determination of (unbounded) convex functions via Crandall-Pazy directions} 
\author{Aris Daniilidis, David Salas, Sebasti\'{a}n Tapia-Garc\'{i}a}

\begin{document}

\maketitle

\begin{abstract}
\noindent It has been recently discovered that a convex function can be determined by its slopes and its infimum value, provided this latter is finite. The result was extended to nonconvex functions by replacing the infimum value by the set of all critical and asymptotically critical values. In all these results boundedness from below plays a crucial role and is generally admitted to be a paramount assumption. Nonetheless, this work develops a new technique that allows to also determine a large class of unbounded from below convex functions, by means of a Neumann-type condition related to the Crandall-Pazy direction.

\end{abstract}
\tableofcontents
\setlength{\parindent}{0pt}

\section{Introduction}\label{sec:intro}

In 2018, an unexpected result at the time was presented in \cite{BCD2018}: For every two convex functions $f,g:\HH\to\R$ over a Hilbert space $\HH$ that are bounded from below and $\mathcal{C}^2$-smooth, one has that
\begin{equation}\label{eq:FirstDeterminationResult}
\forall x\in\HH,\,\,\|\nabla f(x)\| = \|\nabla g(x)\| \quad\implies\quad f= g + \text{cst}.
\end{equation}
In other words, only scalar first-order information (the norm of the gradient) is needed to determine such functions up to an additive constant. This result was extended in \cite{PSV2021} to convex functions over Hilbert spaces that are bounded from below and lower semicontinuous. The scalar first-order information is replaced by the distance of zero to the convex subdifferential which coincides in the convex case with the (metric) slope. The slope of a function $f:\HH\to\R\cup\{+\infty\}$  at a point $x\in\HH$ is given by
\begin{equation}\label{eq:slope-Def}
    s_f(x) = \limsup_{y\to x} \frac{\max\{ f(x) - f(y), 0 \}}{d(x,y)}.
\end{equation}
This concept was first introduced in \cite{GMT1980} to study gradient flow dynamics in metric spaces. The determination result of \cite{PSV2021} can be written as follows: For every two convex functions $f,g:\HH\to\R$ over a Hilbert space $\HH$ that are bounded from below and continuous one has that
\begin{equation}\label{eq:FirstDeterminationResult-slope}
\forall x\in\HH,\,\,s_f(x) = s_g(x) \quad\implies\quad f= g + \text{cst}.
\end{equation}
The idea of the proof of \cite{PSV2021} is very simple: First, for a given point $x\in \HH$, one considers the curve $\gamma:[0,+\infty)\to \HH$ given by the (unique) solution of the differential inclusion
\begin{equation}\label{eq:SubgradientFlow-Intro}
\begin{cases}
    \dot{\gamma}(t) \in - \partial f(\gamma(t)),\quad \text{a.e. }t\geq 0,\\
    \gamma(0) = x.
\end{cases}
\end{equation}
Then, convexity of $f$ and $g$, together with the fact that $\int_{0}^{+\infty} s_g(\gamma (t))^2dt = \int_{0}^{+\infty} s_f(\gamma (t))^2dt <+\infty$ entail that $f(\gamma(t))\to \inf f$ and $g(\gamma(t))\to \inf g$ (see, e.g., the discussion in \cite{DLS2024}). Then, a simple computation using chain rule of convex functions (see, e.g., \cite[Proposition 17.2.5]{ABM2014-book}) together with Cauchy-Schwartz inequality, yields that
\[
\frac{d}{dt} (f-g)(\gamma(t)) \leq s_g(x)^2 - s_f(x)^2 = 0.
\]
One concludes that $f-g$ is nonincreasing along $\gamma$, and consequently that $f(x) \geq g(x) + \inf f - \inf g$. The reverse inequality follows by exchanging the roles of $f$ and $g$ in the development. In a nutshell: we follow the subgradient flow of $f$ to conclude that $f \geq g + \inf f - \inf g$, and then we follow the subgradient flow of $g$ to conclude the reverse inequality. This works since the subgradient flow of $f$ (respectively the subgradient flow of $g$) brings both $f$ and $g$  to their infimal values, whether these values are attained in $\HH$ or are asymptotically reached ``at infinity''.

The determination result of \cite{PSV2021} was extended to arbitrary Banach spaces in \cite{TZ2023}. In the nonconvex setting, similar results using the metric slope were derived for continuous functions in metric spaces: In \cite{DS2022}, functions are considered to be inf-compact and a boundary condition on the set of critical points where the slope is zero is added; In \cite{DLS2024}, the result is derived for continuous functions in complete metric spaces by including also boundary conditions at asymptotically critical sequences; In \cite{DMS2024,DLS2024} the determination result is also extended to the case where the slope is replaced by an abstract notion of scalar first-order information, called \textit{descent modulus}. In \cite{Vilches2021}, a similar result to~\eqref{eq:FirstDeterminationResult-slope} is obtained using proximal operators; In \cite{IvanovZlateva2025} alternative proofs of the determination results of \cite{DMS2024,DLS2024} are derived, based on Ekeland's Variational Principle. Stability results for the slope have been recently investigated in \cite{DD2023, DST2024, LT2024,PTZ2025}.

In all these results, the hypothesis of boundedness from below has been paramount. In the convex case, the constant of~\eqref{eq:FirstDeterminationResult} and~\eqref{eq:FirstDeterminationResult-slope} is in fact the difference of the infimum values, $\inf f - \inf g$. The simple counterexample $f(t) = t$ and $g(t) = -t$ shows that~\eqref{eq:FirstDeterminationResult-slope} cannot holds for unbounded convex functions, even in the one-dimensional case.

Notice that~\eqref{eq:FirstDeterminationResult-slope} can be read as a uniqueness result: for every two convex functions $f,g:\HH\to\R$ over a Hilbert space $\HH$ that are bounded from below and lower semicontinuous, one has that
\begin{equation}\label{eq:DeterminationAsUniqueness}
\left.\begin{array}{r}
s_f(x) = s_g(x),\,\,\forall x\in\HH, \\
\inf f = \inf g
\end{array}\right\} \quad\implies\quad f= g.
\end{equation} 
That is, for a lower semicontinuous function $\ell:\mathcal{H}\to \R_+\cup\{+\infty\}$ and a constant $c\in\R$, there is at most one convex lower semicontinuous solution to the eikonal equation
\begin{equation}\label{eq:Eikonal-with-Inf}
\begin{cases}
    s_u(x) = \ell(x),\quad x\in \HH,\\
    \inf u = c.
\end{cases}
\end{equation}
Eikonal equations and Hamilton-Jacobi equations using slopes have been recently studied in \cite{Gangbo2015,LP2025,LSZ2021,LSZ2025,LZ2023} to extend the theory of viscosity solutions to metric spaces. In our setting, however, the interesting flavor of~\eqref{eq:Eikonal-with-Inf} is that the domain of the equation is the whole space and the value $c=\inf u$ is playing the role of a Dirichlet-type boundary condition at infinity: the value at which all the subgradient flows converge.

The goal of this work is to replace the Dirichlet-type boundary condition $\inf u = c$ by another boundary condition at infinity that also captures the case of unbounded convex functions. We achieve this by studying the gradient flow of convex functions. We note that whenever $\gamma:[0,+\infty)\to\HH$ is a subgradient curve of a convex function $f$ (i.e., solution of~\eqref{eq:SubgradientFlow-Intro} for some $x\in\mathcal{H}$), then $\gamma$ verifies that
\begin{equation}\label{eq:Convergence-SubgradientFlow}
\dot{\gamma}(t) \xrightarrow{t\to +\infty} -p_f,
\end{equation}
where $p_f$ is a unique vector associated to the range of the subdifferential $\partial f$. The vector $p_f$ relates to the seminal works of Crandall and Pazy \cite{cp1969,Pazy1971,p1978} on the asymptotic behavior of semigroups of contractions and will be further called \textit{the Crandall-Pazy direction} associated to~$f$ (see Definition~\ref{def:Crandall-PazyDir} below). We then use this direction to study the following question: Let $f,g:\HH\to\R\cup\{+\infty\}$ be two convex and lower semicontinuous functions (not necessarily bounded from below). Does it hold

\begin{equation}\label{eq:Conjecture-DeterminationNeumann}
\left.\begin{array}{r}
s_f(x) = s_g(x),\,\,\forall x\in\HH, \\
p_f = p_g
\end{array}\right\} \quad{\implies}\quad f= g +\text{cst}~~?
\end{equation}
Note that, whenever implication~\eqref{eq:Conjecture-DeterminationNeumann} holds, the Crandall-Pazy direction $p_f$ is acting as a Neumman-type condition at infinity, replacing the Dirichlet-type paradigm of the literature that was restricting (up to now) the whole study to the class of functions that are bounded from below.

The main result of this work is to establish~\eqref{eq:Conjecture-DeterminationNeumann} for convex functions of class $\mathcal{C}^{1,1}_{\rm loc}$ (Theorem~\ref{thm:DeterminationC11-General}). We also show that the result holds for convex lower semicontinuous functions for which the Crandall-Pazy direction is attained (Proposition~\ref{prop:Determination-nonsmooth}). It is still unknown if~\eqref{eq:Conjecture-DeterminationNeumann} holds in full generality.

\subsection{Preliminaries}
In what follows, we will always work on a Hilbert space $\HH$, with inner product $\langle\cdot,\cdot\rangle$ and induced norm $\|\cdot\|$.

We assume that the reader is familiar with convex subdifferentials and maximally monotone operators. We refer the reader to \cite{Brezis1973Operateurs,Phelps1993Convex} for a detailed exposition. For a maximally monotone operator $A:\HH\tto \HH$ we denote its range by $A(\HH)$ and its domain by $\dom(A)$. In the particular case of the convex subdifferential of a convex function $\partial f:\HH\tto\HH$, we denote by $\partial^{\circ} f(x)$ the element of minimal norm of $\partial f(x)$, whenever $\partial f(x)$ is nonempty. That is,
\[
\forall x\in \dom (\partial f),\quad \partial^{\circ} f(x) := \proj_{\partial f(x)}(0).
\]

We will also assume that the reader is familiar with the theory of nonexpansive semigroups, which will be used in Section~\ref{sec:CP-Dir}, as well as with the theory of generalized derivatives and Jacobians in the sense of Clarke, which will be used in Section~\ref{sec:Determination}. We refer the reader to \cite{Brezis1973Operateurs} for the former, and to \cite{Clarke1983Optimization} for the latter. To ease the presentation, in the aforementioned sections, we will recall the main elements we will use.

\section{Crandall-Pazy Directions}\label{sec:CP-Dir}
Recall that a nonlinear semigroup $S$ over a convex domain $D\subset \HH$ is a family $(S(t))_{t\geq 0}$ of operators over $D$ onto itself satisfying
\begin{enumerate}
    \item $S(0) = \Id$.
    \item $S(t)S(s) = S(t+s)$, for all $t,s\geq 0$.
    \item For every $x\in D$, the map $t\mapsto S(t)x$ is continuous.
\end{enumerate}
The semigroup $S$ is said to be nonexpansive if every $S(t)$ is nonexpansive, that is,
\begin{equation}\label{eq:Contraction}
\|S(t)x-S(t)y\|\leq \|x-y\|, \forall x,y\in D.
\end{equation}
It is well-known (see \cite{cp1969}) that any nonexpansive semigroup $S$ admits a unique maximally monotone (set-valued) operator $A$ as generator: that is, $A$ is the unique maximally monotone operator on $\HH$ verifying that $\overline{\dom A} = D$ and that the curve $t\mapsto S(t)x$ (with $x\in D$) can be constructed as the unique solution of the differential inclusion
\begin{equation}\label{eq:DiffInclusion-MMOp}
\begin{cases}
    \dot{\gamma}(t) \in -A(\gamma(t)),\quad \forall \,\,\text{a.e.}\, \,t\geq 0,\\
    \gamma(0) = x.
\end{cases}
\end{equation}
Let us now consider  $f:\HH\to \R\cup\{+\infty\}$ be a lower semicontinuous convex function. It is well-known that for every $x\in \overline{\dom f}$ there exists a unique absolutely continuous curve ${\gamma_x:[0,+\infty)\to\HH}$ solving the differential inclusion
\begin{equation}\label{eq:SubgradientFlow}
\begin{cases}
    \dot{\gamma}(t) \in -\partial f(\gamma(t)),\quad \forall \,\,\text{a.e.}\, \,t\geq 0,\\
    \gamma(0) = x.
\end{cases}    
\end{equation}
Using these curves, we can construct a nonexpansive semigroup $S_f:=(S_f(t))_{t\geq 0}$ over the domain $D=\overline{\dom f}$, given by $S_f(t)x := \gamma_x(t)$. In this case, the maximally monotone operator $\partial f$ acts as the \emph{generator} of the semigroup. The semigroup $S_f$ is called the \emph{subgradient flow semigroup} of $f$ and the curve $\gamma_x:t\mapsto S_f(t)x$ the respective \emph{subgradient curve} of $f$ emanating from $x\in \overline{\dom f}$. This is a standard construction, and we refer the reader to \cite{Brezis1973Operateurs} or to \cite[Chapter 17]{ABM2014-book} for a modern presentation.

In a series of papers \cite{cp1969,Pazy1971,p1978}, Crandall and Pazy studied the asymptotic behavior of contractions. In particular, in \cite[Theorem~3.9]{p1978} it is shown that for any nonexpansive semigroup $S$ over a convex domain $D\subset \HH$ with generator $A$, one has that
\begin{equation}\label{eq:Pazy-convergence}
    \forall x\in D,\quad \frac{S(t)x}{t} \xrightarrow{t\to+\infty} - p,
\end{equation}
where $p$ is the unique element of minimal norm of the closed convex set $\overline{A(\HH)}$ (the convexity of $\overline{A(\HH)}$ is shown in \cite[Th\'{e}or\`{e}me~2.2]{Brezis1973Operateurs}). The above result of~\eqref{eq:Pazy-convergence} can also be found in an earlier paper by S.~Reich, see \cite[Theorem~3.4]{Reich1974}. The fact that the semigroup $S_f$ fits the setting of the result motivates the following definition. 
\begin{definition}[Crandall-Pazy direction]\label{def:Crandall-PazyDir} Let $f:\mathcal{H}\to\R\cup\{+\infty\}$ be a lower semicontinuous convex function. Recalling that $\overline{\partial f(\HH)}$ is convex, we define the \emph{Crandall-Pazy direction of $f$} as 
    \begin{equation}\label{eq:CP-direction}
     p_f := \mathrm{proj}_{\overline{\partial f(\HH)}}(0).   
    \end{equation} 
In particular, the Crandall-Pazy direction of $f$ verifies that  
\[
\frac{S_f(t)x}{t}\xrightarrow{t\to\infty} -p_f
\]
for every $x~\in~\overline{\dom f}$, where $S_f$ is the subgradient flow semigroup of $f$.    
\end{definition}

\subsection{Relation with slopes and descent sequences}
The slope of a lower semicontinuous convex function $f:\HH\to \R\cup\{+\infty\}$ is characterized as follows:
\begin{equation}\label{eq:slope-distance}
    s_f(x) = \dist(0,\partial f(x)) = \|\partial^{\circ}f(x)\|,
\end{equation}
under the convention that $\dist(0,\emptyset) = +\infty$. It is well known (see, e.g., \cite{AzeCorvellec2004}) that convexity entails that $s_f$ is lower semicontinuous. This yields that
\begin{equation}\label{eq:Inf-sf=norm-pf}
    \inf s_f = \inf\{ \|x^*\|\ :\ x^* \in \partial f(\HH) \} = \|p_f\|.
\end{equation}
Whenever $f$ is bounded from below, any subgradient curve $\gamma$ emanating from some $x\in \overline{\dom f}$ verifies that
\[
s_f(\gamma(t)) = \|\partial^{\circ} f(\gamma(t))\|\xrightarrow{t\to+\infty} 0 = \inf s_f.
\]
The next proposition shows that this convergence holds also for unbounded functions.

\begin{proposition} \label{prop: nonincreasing speed}
    Let $x\in \overline{\textup{dom}}\,f$ and let $\gamma:[0,+\infty)\to \HH$  be the subgradient curve emanating from $x$.
    Then the map $t\mapsto \|\partial^\circ f(\gamma(t))\|$ is nonincreasing and $\|\partial^\circ f(\gamma(t))\|\xrightarrow{t\to+\infty} \inf s_f$.
\end{proposition}
\begin{proof}
    The first assertion is well known and it can be found, for instance, in \cite[Theorem 17.2.3]{ABM2014-book}.
    We prove the second assertion. 
    Let $x_1,x_2\in \HH$ and let $\gamma_i$, with ${i=1,2}$, be the subgradient curve emanating from $x_i$ respectively. 
    Set ${b_i=\lim_{t\to+\infty}\|\partial^\circ f(\gamma_i(t))\|}$, for $i=1,2$, and assume that $b_1>b_2$.
    Without loss of generality, we can assume that $\partial f(x_2)\neq \emptyset$, otherwise we replace the initial point $x_2$ by the point~$\gamma_2(1)$.
    For any $t>0$ and $i\in\{1,2\}$, we have that 
    \[f(\gamma_i(t))=f(x_i)-\int_{0}^t \|\partial^\circ  f(\gamma_i(s)\|^2ds.\] 
    However, by convexity,
    \begin{align*}
        f(x_1)-\int_{0}^t \|\partial^\circ f(\gamma_1(s))\|^2ds& = f(\gamma_{1}(t))\\
        &\geq f(\gamma_2(t))+\langle \partial^\circ  f(\gamma_2(t)),\gamma_1(t)-\gamma_2(t)\rangle\\
        &=f(x_2)-\int_{0}^t \|\partial^\circ  f(\gamma_2(s))\|^2ds+\langle \partial^\circ  f(\gamma_2(t)),\gamma_1(t)-\gamma_2(t)\rangle.
    \end{align*}
    Thus, recalling that the subgradient flow semigroup is nonexpansive and using the first assertion, for any $t>0$ we have that
    \begin{align*}
    \int_0^t \big( \|\partial^\circ  f(\gamma_2(s))\|^2-\|\partial^\circ  f(\gamma_1(s))\|^2\big) ds &\geq f(x_2)-f(x_1)-\|\partial^\circ  f(x_2)\|\|\gamma_1(t)-\gamma_2(t)\|\\
    &\geq f(x_2)-f(x_1)-\|\partial^\circ  f(x_2)\|\|x_1-x_2\|,
    \end{align*}
    which is a contradiction because, since $b_1>b_2$, the left-hand side of the above inequality is not bounded from below. It follows that there is a unique value $b$ such that for every $x\in \overline{\dom f}$
    \[
   \|\partial^{\circ}f(\gamma_x(t))\|\to b.
    \]
    Clearly, $b\geq \dist(0,\partial f(\HH))$, since it is given as the limit of norms of elements of $\partial f(\mathcal{H})$. Using now the first assertion ande~\eqref{eq:Inf-sf=norm-pf} we deduce
    \[
    b \leq \inf_{x\in \overline{\dom f}} \|\partial^{\circ}f(x)\| = \dist(0,\partial f(\HH)).
    \]
    The conclusion follows.
\end{proof}

When $p_f = 0$ (which is the case for bounded functions), the above proposition trivially entails that $\partial^{\circ}f(\gamma(t))\to p_f = 0$ as $t\to+\infty$. We now show that this is also the case for $p_f\neq 0$.

\begin{proposition}\label{prop: asymptotic descent}
Let $f:\HH\to\R\cup\{+\infty\}$ be a proper lower semicontinuous convex function. Then, for any sequence $(x_n)_n\subset \HH$ such that $(s_f(x_n))_n$ converges to $\inf s_f$, $(\partial^\circ  f(x_n))_n$ converges to $p_f$. In particular, for any subgradient curve $\gamma$ of $f$, we have that $\lim_{t\to\infty}\partial^\circ f(\gamma(t))=p_f$.
\end{proposition}
\begin{proof}
    Since $(s_f(x_n))_n$ converges, the sequence $(\partial^\circ  f(x_n))_n$ is bounded. Without losing any generality, we can assume that there is $x^*\in\HH$ such that $\partial^\circ  f(x_n)\rightharpoonup x^*$ (weak convergence). Since $\overline{\partial f(\HH)}$ is closed and convex as the closure of the range of a maximal monotone operator (see \cite[Th\'{e}or\`{e}me 2.2]{Brezis1973Operateurs}), it is weakly closed. Thus, $x^*\in \overline{\partial f(\HH)}$ and so $\|x^*\|\geq \inf s_f$. Since the norm is also weakly lower semicontinuous, we get that
    \[
    \inf s_f \leq \|x^*\| \leq \liminf_{n} \| \partial^{\circ}f(x_n)\| = \lim_n s_f(x_n) = \inf s_f.
    \]
    Then, $\|x^*\| = \inf s_f$ and $x^* = \proj_{\overline{\partial f(\HH)}}(0) = p_f$. We conclude that $\partial^\circ  f(x_n)\rightharpoonup p_f$ and $\|\partial^\circ  f(x_n)\|\to \|p_f\|$, which entails that $\partial^\circ  f(x_n)\to p_f$ (see, e.g., \cite[Proposition 3.32]{Brezis2011Functional}). The proof of the first part of the proposition is then finished.
    The second assertion is now a direct consequence of Proposition~\ref{prop: nonincreasing speed}.
\end{proof}


An interesting direct corollary is that the convex function $f$ becomes affine on the (possibly empty) set $\argmin s_f$, which is the set where $p_f$ is attained.


\subsection{Relation with Thom's conjecture at infinity}

Let $f:\mathcal{U}\subset \R^d\to\R$ be an analytic function defined on a nonempty open bounded set $\mathcal{U}$. Also, consider a gradient curve $\gamma:[0,+\infty)\to\mathcal{U}$ of $f$, that is,

\[\dot{\gamma}(t)=-\nabla f (\gamma(t)),~\text{for all }t>0.\]
The Thom gradient conjecture reads as follows:
\begin{center}
    {\it Suppose that $\gamma(t)\to x_0\in \mathcal{U}$. Then $\gamma$ has a tangent at $x_0$,\\ i.e. the limit of secants $\displaystyle\lim_{t\to+\infty}\dfrac{\gamma(t)-x_0}{\|\gamma(t)-x_0\| }$ exists.}
\end{center}
This conjecture was proved by Kurdyka, Mostowski and Parusi\'{n}ski in \cite{kmp2000}.
Recall that, in Euclidean spaces, a gradient curve of a convex function that attains its minimum value always converges to some minimizer \cite{Bruck1975}.
With this in mind, a similar question was considered in \cite{dhl2022,dls2010} for the gradient flow of convex functions (instead of real-analytic ones). 
However, it is shown in \cite{dls2010} that there is a smooth convex function $f:\R^2\to\R$, with $\argmin\,f=\{0\}$, such that for any $x\in \R^2\setminus \{0\}$, the secants of the gradient curve of $f$ emanating from $x$ do not converge. 
In addition, in \cite{dhl2022}, the construction was further improved: the function $f$ can be taken real-analytic on $\R^2\setminus\{0\}$ and satisfying the {\L}ojasiewicz inequality at $0$.\\

In what follows, we provide a compatible definition of secants for curves $\gamma$ that diverge to infinity, i.e. $\lim_{t\to+\infty}\|\gamma(t)\|=+\infty$. 
\begin{definition}\label{def: cosmic secant}
    Let $\gamma:[0,+\infty)\to\R^d$ be a continuous curve such that $\lim_{t\to+\infty}\|\gamma(t)\|=+\infty$. Denoting by $\Sph$ the unit sphere of $\R^d$, we define
    \[
    \mathrm{Sec}(\gamma,+\infty)= \left\{q\in \Sph:~\liminf_{t\to+\infty} \left\|q-\dfrac{\gamma(0)-\gamma(t)}{\|\gamma(0)-\gamma(t)\| }\right\|=0 \right\}.
    \]
\end{definition}
\begin{remark}
    Note that in Definition~\ref{def: cosmic secant} we can replace $\gamma(0)$ by $\gamma(t_0)$ for any $t_0\geq 0$. The set $\mathrm{Sec}(\gamma,+\infty)$  relates to cosmic convergence \cite{bdm2016,r2018}. 
    Indeed, a curve $\gamma:[0,+\infty)\subset\R^d$, such that $\lim_{t\to +\infty}\|\gamma(t)\|=+\infty$, cosmically converges to $q\in \Sph$ if and only if ${\mathrm{Sec}(\gamma,+\infty)=\{q\}}$. 
\end{remark}
Recall that if $f:\R^d\to\R\cup\{+\infty\}$ is a lower semicontinuous convex function such that $p_f\neq 0$, then every subgradient curve $\gamma$ satisfies $\lim_{t\to+\infty}\|\gamma(t)\|=+\infty$. Moreover, since Proposition~\ref{prop: asymptotic descent} yields for the right-hand derivative $\dot{\gamma}^+$ that $\dot{\gamma}^+(t)=-\partial^\circ f(\gamma(t))\underset{t\to+\infty}{\longrightarrow} -p_f$, we have the following.

\begin{corollary}
    Let $f:\R^d\to\R\cup\{+\infty\}$ be a lower semicontinuous convex function. Assume that $p_f\neq 0$. Then
    \[\mathrm{Sec}(\gamma,+\infty)=\left\{-\dfrac{p_f}{\|p_f\|}\right\}.\]
\end{corollary}
The case $p_f=0$ can occur for both bounded or unbounded from below functions. As stated at the beginning of this subsection, the former case was already considered in \cite{dhl2022,dls2010}.
\begin{theorem}{\cite[Section 7.2]{dls2010}}
    There exists a smooth function $f:\R^2\to\R$, with $\argmin\,f = \{0\}$, such that every gradient curve $\gamma:[0,+\infty)\to\R^2$, with $\gamma(0)\neq 0$, satisfies
    \[\mathrm{Sec}(\gamma,+\infty)=\Sph.\]
\end{theorem}
We treat below the case of an unbounded from below convex function with $p_f=0$.
The following example is inspired by the construction of Ryu in \cite{r2018} where
a nonexpansive operator $T:\R^2\to\R^2$ has been constructed, with minimal displacement vector $0$, such that the discrete dynamic $(T^n x)_n$ does not cosmically converge for any initial point $x\in \R^2$. Let us recall that this work of Ryu answered to the negative a question raised in \cite{bdm2016}.
\begin{proposition}\label{prop: noncosmic convergence}
    There exists a $\mathcal{C}^\infty$-function $f:\R^2\to\R$ such that $p_f=0$, $\inf\,f=-\infty$ and, for any subgradient curve $\gamma$ 
    \[\mathrm{Sec}(\gamma,+\infty)~\text{is not a singleton.}\]
\end{proposition}
    In order to prove the above proposition, we need the following lemma.
\begin{lemma}\label{lemma: 1d construction}
    Let $\phi:\R\to\R$ be a decreasing $\mathcal{C}^\infty$-function such that $\phi> 0$. Assume further that $\int_0^{+\infty} \phi(s)ds= +\infty$. 
    Then, there is a $\mathcal{C}^\infty$-function $\Phi:\R\to\R$ such that its gradient curve $\gamma_{x}:[0,+\infty)\to\R$, emanating from $x\in \R$, satisfies $\dot{\gamma}_x(t)=\phi(x+t)$ for all $t> 0$.
\end{lemma}
\begin{proof}
    Consider first the function $r:\R\to\R$ defined by
    \[r(t):=\int_0^t\phi(s)ds,\quad \text{for all }t\in \R.\]
    Since $\phi>0$, it follows that $r$ is increasing and a $\mathcal{C}^{\infty}$-diffeomorphism.
    Now, define $\Phi:\R\to\R$ by
    \[\Phi\left( r(t)\right):=-\int_0^t\phi(s)^2ds,\quad \text{for all }t\in\R.\]
    Using the change of variables $s = r^{-1}(u)$ in the above integral, standard computations lead to $\Phi(t) = -\int_0^t\phi(r^{-1}(u))du$. It follows that  $\Phi'(t) = - \phi(r^{-1}(t))$, for any $t\in\R$. Therefore, $\Phi'$ is increasing, so that $\Phi$ is convex. Define $\gamma_x:\R\to\R$ as $\gamma_x(t):= r(t+x)$. Since $\Phi'(r(x+t)) = -\phi(x+t)$,  we get that for any $x\in \R$ and $t>0$
    \[
    \Phi'(\gamma_x(t)) = \Phi'(r(x+t)) = -\phi(x+t) = -\dot{\gamma}_x(t).
    \]
This proves the assertion. \end{proof}
\begin{proof}[Proof of Proposition~\ref{prop: noncosmic convergence}]
    Let $\phi,\psi:\R\to\R$ be two  $\mathcal{C}^\infty$-functions such that:
    \begin{itemize}
        \item[$(i)$] $\phi$ and $\psi$ are strictly positive and nonincreasing,
        \item[$(ii)$] $\lim_{t\to+\infty}\phi(t)=\lim_{t\to+\infty}\psi(t)=0$,
        \item[$(iii)$] $\int_0^{+\infty}\phi(s)^2ds=\int_0^{+\infty}\psi(s)^2ds=+\infty$, and
        \item[$(iv)$] $\underset{t\to+\infty}{\limsup}\,\,\dfrac{\int_0^t \phi(s)ds}{\int_0^t \psi(s)ds}\,=\, \,\underset{t\to+\infty}{\limsup}\,\,\dfrac{\int_0^t \psi(s)ds}{\int_0^t \phi(s)ds}\,=\,+\infty$.
    \end{itemize}
\noindent\rule[7pt]{\linewidth}{0.4pt}
Such functions are easy to construct. Consider any sequence $(\alpha_n)_{n\geq 0}\subset (0,+\infty)$, strictly decreasing, converging to $0$ and verifying that $\alpha_n/\alpha_{n+1}\to +\infty$. For example, the sequence can be taken as $\alpha_n = 2^{-n^2}$.

Define now $\phi,\varphi$, and a sequence $(t_n)_{n\in \N}$ as follows:
\begin{itemize}
    \item \textbf{Step 1:} Choose $t_1$ large enough such that $\alpha_1t_1 \geq 1$. For $t\in [0,t_1]$ set $\phi(t) = \alpha_0$ and $\varphi(t) = \alpha_1$.
    \item \textbf{Step 2:} Choose $t_2\geq t_1+1$ large enough such that $\alpha_2(t_2-t_1-1) \geq 1$. For $t\in [t_1+1,t_2]$, set $\phi(t) = \alpha_2$ and extend $\phi$ to $[t_1,t_1+1]$ so that it decreases smoothly from $\alpha_0$ to $\alpha_2$. For $t\in [t_1,t_2]$ set $\varphi(t) = \alpha_1$.
    \item \textbf{Step 3:} Choose $t_3\geq t_2 +1$ large enough such that $\alpha_3(t_3-t_2-1) \geq 1$. For $t\in [t_2+1,t_3]$, set $\varphi(t) = \alpha_3$, and extend $\varphi$ to $t\in [t_2,t_2+1]$ so that it decreases smoothly from $\alpha_1$ to $\alpha_3$. For $t\in [t_2,t_3]$ set $\phi(t) = \alpha_2$.
    \item \textbf{Step 4:} Go back to Step 2 and exchange the roles of $t_1$, $\alpha_1$ and $\alpha_2$ by $t_3$, $\alpha_3$ and $\alpha_4$. 
\end{itemize}
The construction is illustrated in Figure~\ref{fig:construction}, and it clearly verifies conditions $(i)$, $(ii)$ and $(iii)$. 
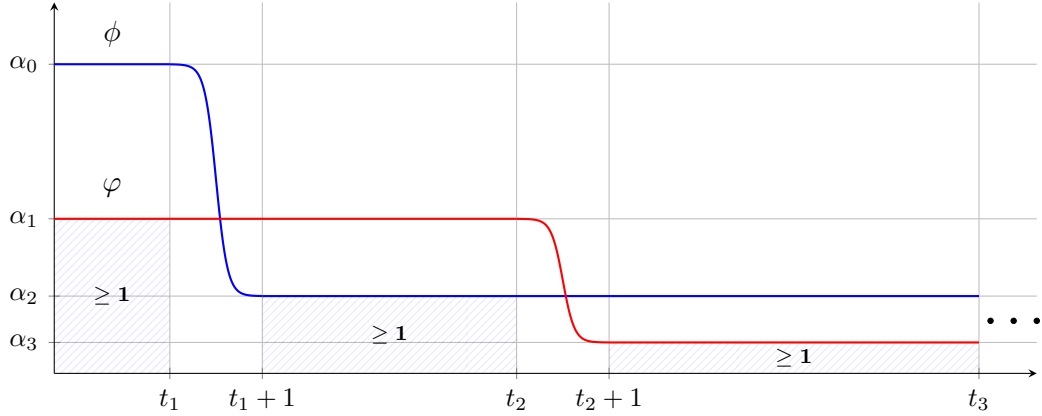
\begin{figure}[ht]
    \centering
\begin{tikzpicture}
    \begin{axis}[
    width=0.9\textwidth,
    height=0.4\textwidth, 
        axis x line=middle,
        axis y line=middle,
        ymin=0, ymax=1.2,
        xmin=0, xmax=8.5,
        samples=200,
        domain=0:8.5,
        xtick={0,1,1.8,4,4.8,8},
        xticklabels={0,\small $t_1$,\small$t_1+1$,\small$t_2$,\small$t_2+1$,\small$t_3$},
        ytick={0,0.1,0.25,0.5,1},
        yticklabels={0,\small $\alpha_3$,\small $\alpha_2$,\small $\alpha_1$, \small $\alpha_0$},
        grid=major,
    ]

         \filldraw[pattern=north east lines, pattern color=blue,draw=none,opacity = 0.3] (axis cs: {0},{0})--(axis cs: {0},{0.5}) -- (axis cs: {1},{0.5})-- (axis cs: {1},{0}) --cycle;

        \node at (axis cs:0.5,0.25) {\scriptsize $\mathbf{\geq 1}$};

        \filldraw[pattern=north east lines, pattern color=blue,draw=none,opacity = 0.3] (axis cs: {1.8},{0})--(axis cs: {1.8},{0.25}) -- (axis cs: {4},{0.25})-- (axis cs: {4},{0}) --cycle;

        \node at (axis cs:2.9,0.125) {\scriptsize $\mathbf{\geq 1}$};
        
        \filldraw[pattern=north east lines, pattern color=blue,draw=none,opacity = 0.3] (axis cs: {4.8},{0})--(axis cs: {4.8},{0.1}) -- (axis cs: {8},{0.1})-- (axis cs: {8},{0}) --cycle;

        \node at (axis cs:6.4,0.05) {\scriptsize $\mathbf{\geq 1}$};
         
        \node at (axis cs:0.5,1.1) {$\phi$};
        \node at (axis cs:0.5,0.6) {$\varphi$};
        \addplot[blue, thick, domain=0:1] {1};
        \addplot[blue, thick, domain=1:1.8] {1 - (3/4)*(1 - tanh(10*(1.4-x)))/2};
        \addplot[blue, thick, domain=1.8:8] {0.25};

        \addplot[red, thick, domain=0:4] {0.5};
        \addplot[red, thick, domain=4:4.8] {0.5 - (0.4)*(1 - tanh(10*(4.4-x)))/2};
        \addplot[red, thick, domain=4.8:8] {0.1};

        \addplot[mark size=1pt,mark=*, only marks] coordinates {(8.1,0.17) (8.3,0.17) (8.5,0.17)};
    \end{axis}
\end{tikzpicture}
    \caption{Construction of $\phi$ and $\varphi$.}
    \label{fig:construction}
\end{figure}

We further adjust the sequence $(t_n)_{n\in\N}$ to ensure that $(iv)$ holds. Indeed, taking $t_n$ sufficiently large each time, we can ensure that
\[
\begin{cases}
\displaystyle\frac{\int_0^{t_n}\phi(t)dt}{\int_0^{t_n}\varphi(t)dt} \approx \frac{\alpha_{n-1}}{\alpha_{n}}\quad\text{if }n\text{ is odd,}\\[2em]
\displaystyle\frac{\int_0^{t_n}\varphi(t)dt}{\int_0^{t_n}\phi(t)dt} \approx \frac{\alpha_{n-1}}{\alpha_{n}}\quad\text{if }n\text{ is even.}
\end{cases}
\]
Therefore, $(iv)$ also holds, completing the construction.

\noindent\rule[7pt]{\linewidth}{0.4pt}\\
Returning to the proof, we deduce from $(iii)$ that
\[ \int_0^{+\infty} \phi(s)ds =\int_0^{+\infty}\psi(s)ds =+\infty.\]
Combining with $(i)$, we apply Lemma~\ref{lemma: 1d construction} to define $\Phi$ and $\Psi$, $\mathcal{C}^\infty$-convex functions associated to $\phi$ and $\psi$ respectively.
Thanks to $(iii)$, the function $\Phi$ and $\Psi$ are not bounded from below. 
Define $f:\R^2\to\R$ by 
\[
f(x,y):=\Phi(x)+\Psi(y),\text{ for all }(x,y)\in \R^2.
\] 
By Lemma~\ref{lemma: 1d construction}, the gradient curve $\gamma=(\gamma_1,\gamma_2):\R\to\R^2$ of $f$, emanating from $0$, satisfies
\[\dot{\gamma}(t)=(\dot{\gamma}_1(t),\dot{\gamma}_2(t))=(\phi(t),\psi(t)).\]
Therefore, thanks to $(ii)$, $\lim_{t\to\infty} \dot{\gamma}(t)=0$. It follows that $p_f=0$.
Finally, thanks to $(iv)$, we readily deduce that
\[\underset{t\to+\infty}{\limsup}\,\,\dfrac{\gamma_1(t)}{\gamma_2(t)}\,=\, \underset{t\to+\infty}{\limsup}\,\,\dfrac{\gamma_2(t)}{\gamma_1(t)}\,=\,+\infty\]
Therefore, $\mathrm{Sec}(\gamma,+\infty)\supset\{(1,0),(0,1)\}.$
Finally, since the subgradient flow defines a nonexpansive semigroup, for any subgradient curve $\nu$ of $f$, we have that $\mathrm{Sec}(\nu,+\infty)=\mathrm{Sec}(\gamma,+\infty).$
\end{proof}

\begin{remark} Recall that if $f:\R^d\to\R$ is a $\mathcal{C}^1$-semi-algebraic function, then for any gradient curve $\gamma$ the secants at infinity stabilize, in the sense that $\mathrm{Sec}(\gamma,+\infty)$ is a singleton  \cite{Grandjean2007,Loi2016}. Therefore, our construction scheme cannot produce a semi-algebraic function. 
    
\end{remark}


\section{Main result: Determination of $\mathcal{C}^{1,1}_{\rm loc}$-convex functions}\label{sec:Determination}


In this section we show that for any two convex functions $f,g:\HH\to\R$ of class $\mathcal{C}^{1,1}_{\rm loc}$, the implication~\eqref{eq:Conjecture-DeterminationNeumann} holds true. 
\begin{theorem}\label{thm:DeterminationC11-General}
Let $f,g:\mathcal{H}\to \R$ be two convex functions of class $\mathcal{C}^{1,1}_{\rm loc}$ such that
\begin{enumerate}
    \item[(i)] $\|\nabla f(x)\| = \|\nabla g(x)\|$ for all $x\in \mathcal{H}$; and
    \item[(ii)] $p_f = p_g$.
\end{enumerate}
Then, $f$ and $g$ are equal up to an additive constant.    
\end{theorem}

The key idea to prove the above theorem is to leverage on the second-order information provided by the condition $s_f = s_g$. We first present a simple proof for the case where the functions are of class $\mathcal{C}^2$. The proof is a simplification of ideas developed in \cite{BCD2018}. Then, we discuss how this proof can be extended to the $\mathcal{C}^{1,1}_{\rm loc}$ case in the finite-dimensional setting, using the theory of generalized Jacobians of Clarke \cite[Chapter 2]{Clarke1983Optimization}. 
Using a suitable extension of Clarke Jacobians due to Thibault \cite{Thibault1982Generalized}, we extend the technique to separable Hilbert spaces. The determination result in a general Hilbert spaces will follow by applying a suitable separable reduction.

\begin{remark} It is worth to mention that in the $\mathcal{C}^{1,1}_{\rm loc}$ case, the arguments of the proof in the finite-dimensional setting are essentially contained in the proof for separable spaces (and in this sense, Subsection~\ref{sec:CaseC11-finiteDim} could have been omitted). However, the proof in finite dimensions is easier to follow since it relies on standard tools of nonsmooth analysis, such as Rademacher theorem and Clarke Jacobians. For this reason, we decided to include this case in the presentation, giving the reader the opportunity to spot exactly the elements of \cite{Thibault1982Generalized} that are needed to extend the result from finite dimensions to separable Hilbert spaces.
\end{remark}	

\subsection{The case of $\mathcal{C}^2$ functions in a Hilbert space}\label{sec:CaseC2}

For a convex function $f:\HH\to\R$ of class $\mathcal{C}^2$, the slope $s_f$ coincides with the mapping ${x\mapsto \|\nabla f(x)\|}$. Moreover, the mapping $x\mapsto \frac{1}{2}\|\nabla f(x)\|^2$ is differentiable and
\[
\nabla\left(\frac{1}{2}\|\nabla f(\cdot)\|^2\right)(x) = H_f(x)\nabla f(x),
\]
where $H_f(x)$ denotes the Hessian matrix of $f$ at $x$. Note that for two convex functions $f$ and $g$, whenever $\|\nabla f\| = \|\nabla g\|$ around a point $x\in\HH$, one has that
    \begin{equation}\label{eq:Jaime-Ortega}
    H_f(x)\nabla f(x) = \nabla \left(\frac{1}{2}\|\nabla f(\cdot)\|^2\right)(x) = \nabla \left(\frac{1}{2}\|\nabla g(\cdot)\|^2\right)(x) = H_g(x)\nabla g(x).
    \end{equation}
    
Recall that the Hessian  of a convex function is always positive semidefinite. With this in mind, we announce the next proposition, establishing~\eqref{eq:Conjecture-DeterminationNeumann} for functions of class $\mathcal{C}^2$.
\begin{proposition}\label{prop:Determination-C2} Let $f,g:\mathcal{H}\to \R$ be two convex functions of class $\mathcal{C}^{2}$ such that
\begin{enumerate}
    \item[(i)] $\|\nabla f(x)\| = \|\nabla g(x)\|$ for all $x\in \mathcal{H}$; and
    \item[(ii)] $p_f = p_g$.
\end{enumerate}
Then, $f$ and $g$ are equal up to an additive constant.
\end{proposition}
\begin{proof}
    Choose $x\in \mathcal{H}$ and let $\gamma:[0,+\infty)\to\mathcal{H}$ be the unique solution of 
    \[
    \begin{cases}
        \dot{\gamma}(t) = -\nabla g(\gamma(t)),\quad\text{ for all }t>0,\\
        \gamma(0) = x.
    \end{cases}
    \]
    Define the function $\phi(t) = \frac{1}{2}\|\nabla f(\gamma(t)) - \nabla g(\gamma(t))\|^2$. Then, omitting the dependency on $t$ to ease the notation, we can write
    \begin{align*}
        \phi' &\,= \langle \nabla f(\gamma) - \nabla g(\gamma), (H_f(\gamma) - H_g(\gamma))\dot{\gamma}\rangle\\[0.2cm]
        &\,= \langle \nabla f(\gamma) - \nabla g(\gamma), H_g(\gamma)\nabla g(\gamma) - H_f(\gamma)\nabla g(\gamma)\rangle\\[0.2cm]
        &\stackrel{\eqref{eq:Jaime-Ortega}}{=} \langle \nabla f(\gamma) - \nabla g(\gamma), H_f(\gamma)\nabla f(\gamma) - H_f(\gamma)\nabla g(\gamma)\rangle\\[0.2cm] 
        &\,= \langle \nabla f(\gamma) - \nabla g(\gamma), H_f(\gamma)(\nabla f(\gamma) - \nabla g(\gamma))\rangle\geq 0.
    \end{align*}
    This yields that $\phi$ is nondecreasing. By Proposition~\ref{prop: asymptotic descent}, we have that $\nabla f(\gamma(t))\to p_f$ and $\nabla g(\gamma(t))\to p_g$ as $t\to+\infty$. Using hypothesis $(ii)$ we get that for all $t>0$
    \[
    0\leq \phi(t) \leq \lim_s \phi(s) = \frac{1}{2}\|p_f - p_g\|^2 = 0.
    \]
    We deduce that $\nabla f(x) = \nabla g(x)$. Since $x\in\mathcal{H}$ is arbitrary, the proof is complete. 
\end{proof}

\subsection{The case of $\mathcal{C}^{1,1}_{\rm loc}$ functions in finite dimensions}\label{sec:CaseC11-finiteDim}

If the convex functions $f$ and $g$ are not longer $\mathcal{C}^2$, we cannot replicate directly the argument of Proposition~\ref{prop:Determination-C2}, since  the functions $\nabla f$ and $\nabla g$ might not be differentiable along the curve $\gamma$. However, when $\mathcal{H} = \R^d$, thanks to the Rademacher theorem the set
\begin{equation}\label{eq:Domain-Diff}
\Omega_f = \{x\in \R^d\,:\, \nabla f\text{ is differentiable at }x\}
\end{equation}
has full Lebesgue measure, provided that $\nabla f$ is locally Lipschitz. This allows to define a generalized Hessian as the convex envelope of limits of derivatives of $\nabla f$.
Generally, for any locally Lipschitz function $F:\R^d\to\R^n$, we can define the generalized Jacobian of $F$ in the sense of Clarke (see, \cite[Section 2.6]{Clarke1983Optimization}) as
\begin{equation}\label{eq:Def-GeneralizedJacobian}
J^CF(x) = \cco\left\{ \lim_{n\to\infty} JF(x_n)\,:\, \Omega_F\ni x_n \to x \right\},
\end{equation}

where $\Omega_F$ is the set of differentiability of $F$ and $JF$ denotes the (usual) Jacobian of $F$ whenever it exists. We use the following three properties of the generalized Jacobian and the Clarke subdifferential (case $n=1$). 
\begin{itemize}
    \item \textbf{Replacement of $\Omega_{F}$} (\cite[Proposition~2.6.4]{Clarke1983Optimization}):\\ For any $\Omega\subset\Omega_{F}$ of full Lebesgue measure, and any $v\in \R^d$, one has that
    \[
        J^CF(x)v = \cco\left\{ Av\,:\, A = \lim_{n\to\infty} JF(x_n),\, \Omega\ni x_n \to x \right\}.
    \]
    \item \textbf{Chain rule for generalized Jacobians} (\cite[Proposition 2.6.6]{Clarke1983Optimization}):\\ If $h:\R^n\to \R$ is a Lipschitz function, then
    \[
    \partial^C(h\circ F)(x) \subset [J^CF(x)]^{\top}\partial^Ch(F(x)) = \{ A^{\top}v\,:\, A \in J^CF(x),\, v\in \partial^Ch(F(x)) \},
    \]
    where $A^{\top}$ is the transpose matrix of $A$.
    \item\textbf{Chain rule for Clarke subdifferential ($n=1$)} (\cite[Theorem 2.3.10]{Clarke1983Optimization}):\\ If $\psi:\R^d\to \R$ is (locally) Lipschitz and $\gamma:[0,+\infty)\to\R^d$ is  a curve of class $\mathcal{C}^1$, then
    \begin{equation}\label{eq:ChainRule-with-curves}
        \partial^C(\psi\circ \gamma)(t) \subset \langle \partial^C\psi(\gamma(t)),\dot{\gamma}(t)\rangle := \{\langle x^*,\dot{\gamma}(t)\rangle\,:\, x^*\in \partial^C\psi(\gamma(t))\}.
    \end{equation}
\end{itemize}
With these three key points in mind, we can prove the following proposition.  
\begin{proposition}\label{prop:Determination-C11-finiteDim} Let $f,g:\R^d\to \R$ be two convex functions of class $\mathcal{C}^{1,1}_{\rm loc}$ such that
\begin{enumerate}
    \item[(i)] $\|\nabla f(x)\| = \|\nabla g(x)\|$ for all $x\in \R^d$; and
    \item[(ii)] $p_f = p_g$.
\end{enumerate}
Then, $f$ and $g$ are equal up to an additive constant.
\end{proposition}
\begin{proof}
    Fix $x\in \R^d$ and let $\gamma:[0,+\infty)\to\R^d$ be the unique solution of 
    \begin{equation}\label{eq:Marco-Lopez-Cerda}
    \begin{cases}
        \dot{\gamma}(t) = -\nabla g(\gamma(t)),\quad\text{ for all }t>0,\\
        \gamma(0) = x.
    \end{cases}
    \end{equation}
    Define the (locally Lipschitz) function 
    \begin{equation}\label{eq:Alejandro-Jofre}
    \phi(t) = \frac{1}{2}\|\nabla f(\gamma(t)) - \nabla g(\gamma(t))\|^2.
    \end{equation}
    We only need to show that $\phi'(t)\geq 0$, almost everywhere. Since $\phi'(t) \in \partial^{C}\phi(t)$ whenever $\phi'(t)$ exists \cite[Proposition~2.2.2]{Clarke1983Optimization}, it is enough to prove that $\partial^{C}\phi(t)\subset [0,+\infty)$. \\    
    
    Since $\gamma(\cdot)$ is a function of class $\mathcal{C}^1$,  the chain rule for the Clarke subdifferential~\eqref{eq:ChainRule-with-curves} allows to assert, for any Lipschitz function $h:\R^d\to\R$, the inclusion $\partial^C(h\circ \gamma)(t) \subset \partial^Ch(\gamma(t))\dot{\gamma}(t)$. Then, using also the generalized Jacobian chain rule \cite[Proposition 2.6.6]{Clarke1983Optimization} and setting $u=\gamma(t)$, we have
   \begin{align}
        \partial^C\phi(t) &\subseteq \left\langle \partial^C \left( \frac{1}{2}\|\nabla f-\nabla g\|^2 \right)(\gamma(t)), \dot{\gamma}(t)\right\rangle\notag\\
        &\subseteq \left\langle [J^{C}(\nabla f - \nabla g)(u)]^{\top}(\nabla f(u) - \nabla g(u)),  -\nabla g(u)\right\rangle\notag\\
        &= \Big\langle \nabla f(u) - \nabla g(u),  J^{C}(\nabla f - \nabla g)(u)(-\nabla g(u))\Big\rangle.\label{eq:Jaime-Ortega2}
    \end{align}
    Set $\Omega = \Omega_f\cap \Omega_g$. Then, by \cite[Proposition 2.6.4]{Clarke1983Optimization}, we have
    \[
    J^{C}(\nabla f - \nabla g)(u)(-\nabla g(u)) = \cco\left\{ H(-\nabla g(u))\, :\, \lim_n (H_f - H_g)(x_n),\, \Omega\ni u_n \to u \right\}.
    \]
    Choose $H = \lim_n (H_f - H_g)(u_n)$. Since $\nabla f$ and $\nabla g$ are locally Lipschitz, the sequence ${\{(H_f - H_g)(u_n)\}_n}$ is bounded and 
    \begin{align*}
        H(-\nabla g(u)) &\,= \lim_n (H_f - H_g)(u_n)(-\nabla g(u_n))\\
        & \stackrel{\eqref{eq:Jaime-Ortega}}{=} \lim_{n} H_f(u_n)(\nabla f(u_n) - \nabla g(u_n))\\
        & \,\in J^C(\nabla f)(u)(\nabla f(u)-\nabla g(u )).
    \end{align*}
   We deduce that 
    \[
    J^{C}(\nabla f - \nabla g)(u)(-\nabla g(u)) \subseteq J^C(\nabla f)(u)(\nabla f(u)-\nabla g(u))
    \]
    and so, using~\eqref{eq:Jaime-Ortega2}, we conclude
    \[
    \partial^C\phi(t) \subseteq \Big\langle \nabla f(u) - \nabla g(u),  J^C(\nabla f)(u)(\nabla f(u)-\nabla g(u))\Big\rangle.
    \]
    Since $J^C(\nabla f)(u)$ contains only positive semidefinite matrices (by convexity of $f$), the conclusion follows. 
\end{proof}

\subsection{The case of $\mathcal{C}^{1,1}_{\rm loc}$ functions in a separable Hilbert space}\label{sec:CaseC11-Separable}

The main obstruction to extend Proposition~\ref{prop:Determination-C11-finiteDim} to infinite dimensions is the fact that Rademacher theorem is no longer valid. For a locally Lipschitz function $\psi:\mathcal{H}\to\R$ it is still possible to define the Clarke subdifferential $\partial^C \psi$ using generalized directional derivatives (see \cite[Chapter 2]{Clarke1983Optimization}), but the definition of generalized Jacobians as in~\eqref{eq:Def-GeneralizedJacobian} is less standard.

In the particular case that $\mathcal{H}$ is a separable Hilbert space and $F:\mathcal{H}\to\mathcal{H}$ is a Lipschitz function, the set
\[
\Omega_{F} = \{ x\in \mathcal{H}\,:\, F\text{ is G\^{a}teaux-differentiable at }x \}
\]
is the complement of a Haar-null set  \cite{Christensen1973Measure} (see also \cite{LindenstraussPreiss2003Frechet} and the references therein for similar results). In particular, $\Omega_{F}$ is dense in $\mathcal{H}$. Then, following \cite{Thibault1982Generalized}, it is possible to extend the notion of generalized Jacobian by defining
\begin{equation}\label{eq:GeneralizedJacobian-Thibault}
\Gamma(F;x) = \cco\left\{\text{w-}\lim JF(x_n)\,:\,\Omega_{F}\ni x_n\to x \right\},
\end{equation}
where the limit $\text{w-}\lim JF(x_n)$ is the limit with respect to the weak-operator topology: that is,
\begin{equation}\label{eq:w-OperatorTopology}
\text{w-}\lim JF(x_n) = A \iff \forall y,z\in \mathcal{H}, \langle y, JF(x_n)z\rangle \to \langle y, Az\rangle.
\end{equation}

In this new setting, the aforementioned properties of the generalized Jacobian $J^CF$ have analogous counterparts for the set-valued operator $\Gamma(F;\cdot)$.

\begin{itemize}
    \item \textbf{Replacement of $\Omega_{F}$} (\cite[Proposition 2.5]{Thibault1982Generalized}): For any $\Omega\subset\Omega_{F}$ such that $\mathcal{H}\setminus\Omega$ is Haar-null, and any $v\in \mathcal{H}$, one has that 
    \[
        \Gamma(F;x)v = \Gamma_{\Omega}(F;x)v,
    \]
    where $\Gamma_{\Omega}(F;x) = \cco\{\text{w-}\lim JF(x_n)\,:\, \Omega\ni x_n\to x \}$.
    \item \textbf{Chain Rule} (\cite[Proposition 2.4]{Thibault1982Generalized}): If $h:\mathcal{H}\to \R$ is  strictly differentiable and $\Omega\subset\Omega_{F}$ such that $\mathcal{H}\setminus\Omega$ is Haar-null, then
    \[
    \partial^C(h\circ F)(x) = [\Gamma_{\Omega}(F;x)]^{\ast}\nabla h(F(x)) = \{ A^{\ast}\nabla h(F(x)) \,:\, A \in \Gamma_{\Omega}(F;x) \},
    \]
    where $A^{\ast}$ denotes the adjoint operator of $A$.
\end{itemize}

\begin{proposition}\label{prop:Determination-C11-Separable} Let $\mathcal{H}$ be a separable Hilbert space. Let $f,g:\mathcal{H}\to \R$ be two convex functions of class $\mathcal{C}^{1,1}_{\rm loc}$ such that
\begin{enumerate}
    \item[(i)] $\|\nabla f(x)\| = \|\nabla g(x)\|$ for all $x\in \mathcal{H}$; and
    \item[(ii)] $p_f = p_g$.
\end{enumerate}
Then, $f$ and $g$ are equal up to an additive constant.
\end{proposition}
\begin{proof}
 Note first that the chain rule for Clarke subdifferential \cite[Theorem 2.3.10]{Clarke1983Optimization} still holds. Choose $\Omega$ as the set of common G\^{a}teaux-differentiability points of $\nabla f$ and $\nabla g$, which is the complement of a Haar-null set (see, e.g., \cite{Thibault1982Generalized} and the references therein). Choose $x\in\mathcal{H}$ and defining $\gamma$ and $\phi$ as in~\eqref{eq:Marco-Lopez-Cerda} and~\eqref{eq:Alejandro-Jofre}, respectively. Using \cite[Proposition 2.4]{Thibault1982Generalized} instead of  \cite[Proposition 2.6.6]{Clarke1983Optimization} and denoting $u=\gamma(t)$, we can replicate the proof of Proposition~\ref{prop:Determination-C11-finiteDim} to deduce
\[
\partial^C\phi(t) \subseteq \langle \nabla f(u) - \nabla g(u), \Gamma_{\Omega}(\nabla f - \nabla g;u) (-\nabla g(u))\rangle.
\]
Then, according to \cite[Proposition 2.5]{Thibault1982Generalized}, we only need to show that
    \[
    \langle \nabla f(u) - \nabla g(u), H(-\nabla g(u))\rangle \geq 0,
    \]
where $H = \text{w-}\lim_n (H_f(u_n) - H_g(u_n))$ with $\Omega\ni u_n\to u$ and $\Omega = \Omega_f\cap\Omega_g$. Since $\nabla f,\nabla g$ are locally Lipschitz, the sequence $(H_f - H_g)(x_n)$ is bounded. Thus, it is not hard to prove that for every $y_n\to y$ and $z_n\to z$ one has that
\[
\langle y,Hz\rangle = \lim_n \langle y_n,(H_f - H_g)(x_n)z_n\rangle.
\]
Then, setting $y_n = \nabla f(u_n)-\nabla g(u_n)$ and $z_n = -\nabla g(u_n)$ and using continuity of $\nabla f$ and $\nabla g$, we can write
\begin{align*}
\langle \nabla f(u) - \nabla g(u), H(-\nabla g(u))\rangle &\,= \lim_n \langle \nabla f(u_n) - \nabla g(u_n), (H_f(u_n) - H_g(u_n))(-\nabla g(u_n))\rangle\\
&\stackrel{\eqref{eq:Jaime-Ortega}}{=}\lim_n \langle \nabla f(u_n) - \nabla g(u_n), H_f(u_n)(\nabla f(u_n)-\nabla g(u_n))\rangle \geq 0.
\end{align*}
The proof is complete.
\end{proof}

\subsection{The case of $\mathcal{C}^{1,1}_{\rm loc}$ functions in a general Hilbert space}\label{sec:CaseC11-General}

In order to derive our main result to convex functions of class $\mathcal{C}^{1,1}_{\rm loc}$ defined on nonseparable Hilbert spaces, we employ a separable reduction technique which reduces the problem to separable spaces.
To do so, we need three lemmas that have independent interest.

The first lemma is well-known. We include a short proof for completeness.
\begin{lemma}\label{lemma: total flow}
Let $\mathcal{H}$ be a Hilbert space and $f:\mathcal{H}\to\R\cup\{+\infty\}$ a lower semicontinuous convex function.
Denote by $S_f:\overline{\mathrm{dom}}\,f\times[0,+\infty)\to\mathcal{H}$ the subgradient flow semigroup of $f$, that is
\[\gamma(\cdot) := S_f(x,\cdot)~\text{is the subgradient curve of $f$ emanating from }x.\]
Then, $S_f$ is continuous.
\end{lemma}
\begin{proof}
    Let $(x_n,t_n)\to (x,t)$. Using the contraction property of the subgradient flow $S_f$, we get
    \begin{align*}
    \|S_f(x_n,t_n)- S_f(x,t)\|&\leq \|S_f(x_n,t_n)- S_f(x,t_n)\| + \|S_f(x,t_n) - S_f(x,t)\|\\
    &\leq \|x-x_n\| + \|S_f(x,t_n) - S_f(x,t)\|\to 0.
    \end{align*}
   This shows the result. 
\end{proof}

The next lemma is the key element for the separable reduction argument.

\begin{lemma}[Separable reduction of gradient flows]\label{lemma: separable reduction}
Let $\mathcal{H}$ be a Hilbert space and let $f_i:\mathcal{H}\to\R$ be a convex continuous function, for every $i\in\N$. Let $A\subset \mathcal{H}$ be a  nonempty separable subset.
Then, there exists a separable subspace $\mathcal{H}_0\subset \mathcal{H}$ such that
\begin{itemize}
    \item[$(i)$] $A\subset \mathcal{H}_0$; and
    \item[$(ii)$] for every $i\in\N$, $S_{f_i}(\mathcal{H}_0,[0,+\infty))= \mathcal{H}_0$.
\end{itemize}
\end{lemma}
\begin{proof}
    If $\mathcal{H}$ is separable there is nothing to prove.
    Without loss of generality, assume that $\mathcal{H}$ is nonseparable.
    
    In what follows, we define inductively two sequences of subsets of $\mathcal{H}$, $\{E_n\}_n$ and $\{F_n\}_n$.
    To start, set $E_1=\mathrm{span}(A)$ and 
    \[
    F_1:= \bigcup_{i\in\N} S_{f_i}(E_1\times[0,+\infty)).
    \]
    Inductively, for $n\geq 2$, we set $E_n:=\mathrm{span}(F_{n-1})$ and
    \[
    F_n:= \bigcup_{i\in\N} S_{f_i}(E_n\times[0,+\infty)).
    \] 
    
    We now prove that for any $n\in\N$, $E_n$ is a separable subspace of $\mathcal{H}$. Indeed, since the vector span of a separable set is separable, it is enough to show that $F_n$ is separable for any $n\in \N$. 
    The proof follows by induction. For $n=1$, we know that $E_1$ is separable. Fix $n\geq 1$ and assume that $E_n$ is separable. Since $S_{f_i}$ is continuous (cf. Lemma~\ref{lemma: total flow}), $F_n$ is a countable union of separable sets. Therefore, $F_n$ is separable and $E_{n+1}=\mathrm{span}(F_n)$ is separable as well.\smallskip\newline   
  Let us denote by $E_\omega$ the subspace $\bigcup_{n\in\N}E_n$. To finish the proof, we show that the subspace $\mathcal{H}_0= \overline{E_\omega}$ is separable and satisfies $(i)$ and $(ii)$.\smallskip\newline
    The separability of $\mathcal{H}_0$ comes from the separability of $E_\omega$, which is a countable union of separable subspaces.
    By construction, $A\subset E_0\subset \mathcal{H}_0$.
    So, we only need to prove $(ii)$. 
    Let $x\in \mathcal{H}_0$ and $i\in \N$. 
    If $x\in E_\omega$, then there is $n\in \N$ such that $x\in E_n$. 
    Therefore, by construction, $S_{f_i}(x,[0,+\infty))\subset E_{n+1}\subset \mathcal{H}_0$. 
    On the other hand, if $x\in \mathcal{H}_0\setminus E_\omega$, there is a sequence $\{x_n\}_n\subset E_\omega$, convergent to $x$.
    Therefore, $S_{f_i}(x_n,[0,+\infty))\subset E_\omega$ for all $n\in \N$.
    The contraction property of the subgradient flows of convex functions leads us to
    \begin{align*}
        \| S_{f_i}(x_n,t)-S_{f_i}(x,t)\|\leq \| x_n-x\| \xrightarrow{n\to\infty} 0,~\quad \text{for any }t\geq 0.
    \end{align*}
    This shows that $S_{f_i}(x,[0,+\infty))\subset \mathcal{H}_0$. The proof is now complete.
\end{proof}

\begin{lemma}\label{lemma: slope reduction}
    Let $\mathcal{H}$ be a Hilbert space and let $f:\mathcal{H}\to\R$ be a convex continuous function. 
    Let $\mathcal{H}_0$ be a closed subspace of $\mathcal{H}$ which is invariant under the subgradient flow of $f$. 
    Then
    \[ s[f](x)=s[f|_{\mathcal{H}_0}](x),\quad\text{for all }x\in \mathcal{H}_0,\]
    where $f|_{\mathcal{H}_0}$ denotes the restriction of $f$ to the space $\mathcal{H}_0$.
\end{lemma}
\begin{proof}
    Let $x\in  \mathcal{H}_0$. By the very definition of the slope as a limit superior, we obtain
    \[s[f](x)=\limsup_{ \mathcal{H}\ni y\to x }\frac{(f(x)-f(y))_+}{\|x-y\|}\geq \limsup_{\mathcal{H}_0\ni y\to x}\frac{(f(x)-f(y))_+}{\|x-y\|}=s[f|_{\mathcal{H}_0}](x). \]
    Let $\gamma_x$ be the subgradient curve of $f$ emanating from $x$. 
    Thanks to \cite[Theorem 17.2.2]{ABM2014-book}, we have that 
    \[
    \partial^\circ f(x)= -\gamma_x^+(0)\in \mathcal{H}_0.
    \]
    
     If $s[f](x)=0$, there is nothing to do. 
     So, assume that $s[f](x)>0$ and set $v=-\partial^{\circ}f(x)/\|\partial^{\circ}f(x)\|$. 
     Since $x,v\in \mathcal{H}_0$, we have that
    \begin{align*}
    s[f|_{\mathcal{H}_0}](x) \geq -(f|_{\mathcal{H}_0})'(x;v) &= -f'(x;v)=-\max_{\xi\in\partial f(x)}\langle v,\xi\rangle = \|\partial^{\circ} f(x)\| = s[f](x),
    \end{align*}
    where $(f|_{\mathcal{H}_0})'(x;\cdot)$ and $f'(x;\cdot)$ are the directional derivatives at $x$ of $f|_{\mathcal{H}_0}$ and $f$ respectively. 
    %
%
\end{proof}
Finally, we can prove the main result of this section, Theorem~\ref{thm:DeterminationC11-General}.
\begin{proof}[Proof of Theorem~\ref{thm:DeterminationC11-General}]
    If $\mathcal{H}$ is a separable Hilbert space, the result follows from Proposition~\ref{prop:Determination-C11-Separable}.
    Let us assume that $\mathcal{H}$ is nonseparable.
    Let $S_f$ and $S_g$ be the subgradient flow semigroup of $f$ and $g$ respectively.
    Consider the following family of subspaces of $\mathcal{H}$:
    \begin{align*}
      \mathscr{H}:=\left\{\begin{array}{c}X~\textit{is a closed and separable subspace}~\text{such that } \\[0.1cm]
      ~S_f(X\times[0,+\infty))\cup S_g(X\times[0,+\infty))\subset X\end{array}\right\}.
    \end{align*}
    Thanks to Lemma~\ref{lemma: separable reduction}, $\mathscr{H}$ is nonempty. In addition
    \begin{align}\label{eq: separable cover}
      \mathcal{H}=\bigcup\{X:~X\in \mathscr{H}\}.  
    \end{align}
    Further, thanks to Lemma~\ref{lemma: slope reduction}, for any $X\in \mathscr{H}$ we have that $p_f\in X$ and $\| \nabla f|_X\| =\|\nabla g|_X\|$.\\
    
    Fix $X\in \mathscr{H}$. 
    By Proposition~\ref{prop:Determination-C11-Separable}, there exists $c_X\in \R$ such that $f=g+c_X$ on $X$.
    Since $0\in X$, we get that $c_X= f(0)-g(0)$.
    Since $X\in \mathscr{H}$ is arbitrary and~\eqref{eq: separable cover} holds, we deduce that $f=g + (f(0)-g(0))$ on $\mathcal{H}$. 
    \end{proof}

\section{The case where Crandall-Pazy direction is attained}

In the previous section, we leverage on the $\mathcal{C}^{1,1}_{\rm loc}$ structure to control the function $\|\nabla f- \nabla g\|$ along the gradient curves using second-order information. It is not clear if this approach can be extended for nonsmooth convex functions. This being said, we show that whenever the Crandall-Pazy direction is attained, the asymptotic behavior can be controlled via a common gradient curve of $f$ and $g$. 

\begin{proposition}\label{prop:Determination-nonsmooth} Let $f,g:\mathcal{H}\to \R\cup\{+\infty\}$ be two lower semicontinuous convex functions such that
\begin{enumerate}
    \item[$(i)$] $s_f(x) = s_g(x)$ for all $x\in \mathcal{H}$;
    \item[$(ii)$] $p_f=p_g$; and
    \item[$(iii)$] $p_f\in\partial f(\HH)$. 
\end{enumerate}
Then, $f$ and $g$ are equal up to an additive constant.
\end{proposition} 
For the sake of clarity, we first prove the case when $f$ and $g$ are of class $\mathcal{C}^1$. Then, by means of some technical lemmas we extend the result to the nonsmooth setting.
Let us start with the following lemma.
\begin{lemma}\label{lemma:1}
    Let $f:\mathcal{H}\to\R\cup\{+\infty\}$ be a lower semicontinuous convex function. 
    Assume that there is $\hat{x}\in \HH$ such that $\partial^\circ f(\hat{x})=p_f$. 
    Then the curve $\eta(t) = \hat{x}-tp_f$, $t\in[0,+\infty)$, is the subgradient curve of $f$ emanating from $\hat{x}$.
\end{lemma}
\begin{proof}
    Let us denote by $\eta$ the subgradient curve of $f$ emanating from $\hat{x}$.  
    By \cite[Theorem~17.2.2]{ABM2014-book}, we know that $\dot{\eta}^+(t)=-\partial^\circ f(\eta(t))$ for all $t\geq 0$ and that $\left\|\dot{\eta}^+(t)\right\|$ is nonincreasing. 
    Since $p_f$ is the unique minimizer for the norm over the closed convex set $\overline{\partial f(\mathcal{H})}$, we deduce that $\dot{\eta}^+(t)=-p_f$ for all $t\geq0$.
    Hence, $\eta(t)=\hat{x}-tp_f$ for all $t\geq 0$.
\end{proof}
Let us prove Proposition~\ref{prop:Determination-nonsmooth} for the case where the functions $f$ and $g$ are of class $\mathcal{C}^1$.
\begin{proposition}\label{prop: case C1 p}
    Let $f,g:\mathcal{H}\to\R$ be two $\mathcal{C}^1$-smooth convex functions. Assume 
    \begin{enumerate}
    \item[$(i)$] $\|\nabla f(x)\| = \|\nabla g(x)\|$ for all $x\in \mathcal{H}$;
    \item[$(ii)$] $p_f=p_g$; and
    \item[$(iii)$] $p_f = \nabla f(\hat{x})$ for some $\hat{x}\in \HH$. 
\end{enumerate}
    Then, $f$ and $g$ are equal up to an additive constant.
\end{proposition}

\begin{proof}
    Let us denote by $p$ the vector $p_f$ (and $p_g$).
    Proposition~\ref{prop: asymptotic descent} entails that $\nabla g(\hat{x})=p$. 
    By Lemma~\ref{lemma:1}, the curve $\eta:[0,+\infty)\to \mathcal{H}$ defined by 
    \[
    \eta(t)=\hat{x}-tp
    \] 
    is the common gradient flow of $f$ and of $g$ emanating from $\hat{x}$.
    Denote by $\psi:\HH\to\R$ the function 
    \[
    \psi:=f-g.
    \]
    Since $\eta$ is a common gradient curve for the functions $f$ and $g$, Proposition~\ref{prop: nonincreasing speed} entails that $\nabla f(\eta(t)) = \nabla g(\eta(t)) = -p$. Thus, $(\psi\circ\eta)'(t) = 0$ for all $t>0$, yielding
    \begin{equation}\label{eq:Sebastian-Tapia-1}
        \psi(\eta(t))= \psi(\eta(0))=\psi (\hat{x}),\text{ for all }t\geq 0.
    \end{equation}
    
    We shall show that $\psi$ is constant on $\HH$, which is readily equivalent to the statement of the proposition. To this end, let $y,z\in \mathcal{H}$ and assume that $\psi(y)\leq \psi(z)$. 
    Consider the gradient curve $\gamma:[0,+\infty)\to \HH$ of $f$ emanating from $y$ and the gradient curve $\nu:[0,+\infty)\to \HH$ of $g$ emanating from $z$. That is, 
    \begin{align*}
        \begin{cases}
            \gamma(0)&=y\\
            \dot{\gamma}(t)&=-\nabla f(\gamma(t)),
        \end{cases}\quad\text{and}\quad \begin{cases}
            \nu(0)&=z\\
            \dot{\nu}(t)&=-\nabla g(\nu(t)).
        \end{cases}
    \end{align*}
    for every $t\geq 0$.  Notice that $(\psi\circ \gamma)'(t)= -\|\nabla f(\gamma(t))\|^2+\langle \nabla g(\gamma(t)),\nabla f(\gamma(t))\rangle \leq 0$ for all $t>0$. 
    A similar computation shows that $(\psi\circ \nu)'(t)\geq 0$ for all $t\geq 0$. Therefore,
    
    \begin{equation}\label{eq:Sebastian-Tapia-2}
    \psi(\gamma(t))\leq \psi(y)\leq \psi(z)\leq \psi(\nu(t)),~\text{for all }t\geq 0.
    \end{equation}
    
    For any $t\geq 0$, consider the function
    \[
    w_t(\lambda)=S_f(\lambda \hat{x}+(1-\lambda)y,t),\quad\lambda\in [0,1].
    \]
    Since the gradient flow semigroup $S_f$ is nonexpansive, the functions $w_t:[0,1]\to \HH$ is $d(\hat{x},y)$-Lipschitz for every $t\geq 0$. We deduce

    \begin{align*}
        \psi(\hat{x})-\psi(y)&\stackrel{\eqref{eq:Sebastian-Tapia-2}}{\leq} \psi(\hat{x})-\psi(\gamma(t))= \psi(\eta(t))-\psi(\gamma(t))\\[0.15cm]
        &\,= \psi(w_t(1)) -\psi(w_t(0))\\
        &\stackrel{\eqref{eq:Sebastian-Tapia-1}}{=} \int_0^1\langle (\nabla f-\nabla g)(w_t(\lambda)),\dot{w}_t(\lambda)\rangle d\lambda\\
        &\,\leq d(\hat{x},y)\int_0^1\big(\|\nabla f(w_t(\lambda))-p \| +\|\nabla g(w_t(\lambda))-p\| \big)d\lambda. 
    \end{align*}
    By Proposition~\ref{prop: asymptotic descent}, $\nabla f (w_t(\lambda)) $ converges to $p$ as $t$ goes to infinity.
    On the other hand, since $t\mapsto\|\nabla f(w_t(\lambda))\|$ is decreasing for every $\lambda\in[0,1]$ (see, e.g. \cite[Theorem~17.2.2]{ABM2014-book}), we have 
    \begin{equation}\label{ eq: domination}
    \sup_{t\geq 0,\lambda\in [0,1]}\{\|\nabla f(w_t(\lambda))-p \| +\|\nabla g(w_t(\lambda))-p\|\}\leq
    \underbrace{\max_{u\in [y,\hat{x}]}\{\|\nabla f(u)\|+\|\nabla g(u)\|+2\|p\|\}}_{=K<+\infty}  
    \end{equation}

    Therefore, we can apply Lebesgue Dominated Convergence theorem to get
    \begin{align*}
    \psi(\hat{x})-\psi (y) &\leq \limsup_{t\to+\infty} \psi (\eta(t))-\psi(\gamma(t))\\
    &\leq d(\hat{x},y)\limsup_{t\to+\infty}\int_0^1\big(\|\nabla f(w_t(\lambda))-p \| +\|\nabla g(w_t(\lambda))-p\| \big)d\lambda = 0.
    \end{align*}
    It follows that $\psi(\hat{x})\leq \psi(y)$. In an analogous way, using the gradient curve $\nu$ of $g$ emanating from $z$, we deduce that
    \[
    \psi(z)-\psi(\hat{x})\leq \limsup_{t\to+\infty} \psi (\nu(t))-\psi(\eta(t))\leq  0.
    \]
    It follows that $\psi(z) \leq \psi(y)$ and consequently the equality holds.
    Since $y$ and $z$ are arbitrary vectors in $\HH$, then $\psi= f-g$ is a constant function. The proof is complete.
\end{proof}

Most of the steps of the above proof can be replicated in the nonsmooth case without much difficulty.
However, some extra work is still required to justify a correct application of the Lebesgue Dominated Convergence theorem. 

The following lemma is known. We include a short proof for completeness.
\begin{lemma}\label{lemma: 2}
    Let $f:\mathcal{H}\to\R\cup\{+\infty\}$ be a proper lower semicontinuous convex function. 
    Let $\ell:[0,1]\to \mathcal{H}$ be a continuous curve such that $ \ell([0,1])\subset \mathrm{dom}\,\partial f$. 
    Then, $\partial^\circ f\circ \ell$ is measurable. 
\end{lemma}
\begin{proof}
    This is a direct application of the Moreau-Yosida regularization. 
    Indeed, let $\lambda>0$ and $f_\lambda:\mathcal{H}\to \R$ be the Moreau-Yosida regularization of $f$. 
    Then, thanks to \cite[Proposition 17.2.2]{ABM2014-book}, $f_\lambda$ is a $\mathcal{C}^{1,1}$-smooth convex function such that 
    \[\lim_{\lambda\to 0} \nabla f_\lambda (x)= \partial^\circ f(x),~\text{for all }x\in \mathrm{dom}\,\partial f.\]
    Therefore, $\partial^\circ f\circ \ell$ is the pointwise limit of the sequence of continuous curves $\left\{\nabla f_{\frac{1}{n}}\circ \ell\right\}_{n\geq 1}$.
    \end{proof}
    Following the notation of the proof of Proposition~\ref{prop: case C1 p}, we derive the following lemma.
\begin{lemma}\label{lemma: 3}
Let $f:\mathcal{H}\to\R\cup\{+\infty\}$ be a proper lower semicontinuous convex function. Let $x,y\in \mathrm{dom}\,f$ and consider, for any $t\geq 0$ and $\lambda\in[0,1]$, $w_t(\lambda)=S_f(\lambda x+(1-\lambda)y,t)$. 
Then,
\[\lim_{t\to+\infty} \sup\{\|\partial^\circ f(w_t(\lambda))\|:~\lambda\in [0,1]\}=\|p\|.\]
\end{lemma}
\begin{proof}
    Since $f$ is a convex function and $[x,y]\subset \mathrm{dom}\,f$, \cite[Theorem 17.2.3]{ABM2014-book} implies that, for any $t>0$, $w_t([0,1])\subset \mathrm{dom}\,\partial f$.
    Recall also that for every $t\geq 0$ the function $w_t$ is $d(x,y)$-Lipschitz and that the map $t\mapsto \|\partial^\circ f(w_t(\lambda))\|$ is nonincreasing for every $\lambda\in [0,1]$. Finally, recall that
    \begin{equation}\label{eq:Miguel-Goberna}
     f(w_t(\lambda))=f(w_0(\lambda))-\int_0^t\|\partial^\circ f(w_s(\lambda))\|^2ds.
    \end{equation}

    Reasoning towards a contradiction, assume that there is $\sigma>0$ such that
    \[
    \sup\{\|\partial^\circ f(w_t(\lambda))\|:~\lambda\in [0,1]\}>\|p\|+\sigma,\quad\text{for all }t>0.
    \]
    Then, for every $t>0$, fix $\lambda_t\in [0,1]$ such that $\|\partial^\circ f(w_t(\lambda_t))\|>\|p\|+\sigma$. 
    Consequently, ${\|\partial^\circ f(w_s(\lambda_t))\|>\|p\|+\sigma}$ for every $s\in(0,t)$. 
    Applying the inequality of the convex subdifferential, we obtain
    \begin{align*}
         f(w_t(\lambda_t))
        &\geq f(w_t(0))+\langle \partial^\circ f(w_t(0)),w_{t}(\lambda_n)-w_t(0)\rangle\\
        & \geq f(w_0(0))-\int_0^t\|\partial^\circ f(w_s(0))\|^2ds- \|\partial^\circ f(w_t(0))\| \|x-y\|.
    \end{align*}

    The above expression together with~\eqref{eq:Miguel-Goberna} for $\lambda = \lambda_t$ yields
    \begin{align*}
        \int_0^t\Big(\|\partial^\circ f(w_s(0))\|^2-\|\partial^\circ f(w_s(\lambda_t))\|^2\Big)ds \geq f(w_0(0))-f(w_0(\lambda_t))-\|\partial^\circ f(w_t(0))\| \|x-y\|.
    \end{align*}
    Thus, for every $t>1$ we obtain
    \begin{align}\label{eq: bound}
        \int_0^t \Big( \|\partial^\circ f(w_s(0))\|^2-\|\partial^\circ f(w_s(\lambda_t))\|^2 \Big)ds \geq f(w_0(0))-\max\{f(x),f(y)\}-\|\partial^\circ f(w_1(0))\| \|x-y\|.
    \end{align}
    Notice that the right-hand side of~\eqref{eq: bound} is a finite number.
    This easily leads to a contradiction since the map $t\mapsto \|\partial^\circ f(w_s(0))\|$ converges to $\|p\|$ and $\|\partial^\circ f(w_s(\lambda_t))\|>\|p\|+\sigma$ for all $s\in[0,t]$. The proof is complete.
\end{proof}

Now we are ready to prove Proposition~\ref{prop:Determination-nonsmooth}
\begin{proof}[Proof of Proposition~\ref{prop:Determination-nonsmooth}]
We follow the same lines of proof of Proposition~\ref{prop: case C1 p}, and we present only  a sketch of the proof highlighting the main differences.\\

Set $\Omega:=\mathrm{dom}\, \partial f= \mathrm{dom}\,
\partial g\subset \HH$, and
define the function $\psi:\Omega\to \R$ by
\[
\psi:=f-g.
\]
Note that $\Omega\subset\mathrm{dom}\,f\cap \mathrm{dom}\,g$.
Let $y,z\in \Omega$. Without loss of generality $\psi(y)\leq \psi(z)$.
As in the proof of Proposition~\ref{prop: case C1 p}, let $p$ be the common value of $p_f$ and $p_g$, and let $\hat{x}\in \Omega$ such that $p\in \partial f(\hat{x})$.  Let $\eta(t) = \hat{x}-tp$ the common gradient curve of $f$ and $g$ emmanating from $\hat{x}$. Define $\gamma$ and $\nu$ as follows:
    \begin{align*}
        \begin{cases}
            \gamma(0)&=y\\
            \dot{\gamma}(t)&\in-\partial f(\gamma(t)),
        \end{cases}\quad\text{and}\quad \begin{cases}
            \nu(0)&=z\\
            \dot{\nu}(t)&\in-\partial g(\nu(t)),
        \end{cases}
    \end{align*}
 for almost every $t\geq 0$.
Thanks to \cite[Theorem 17.2.2]{ABM2014-book}, we have
\[S_f([y,\hat{x}]\times(0,+\infty))\subset \Omega.\] 
For $t\geq0$, consider the function

\[
w_t(\lambda):= S_f(\lambda \hat{x}+(1-\lambda)y,t),\quad \lambda\in[0,1].
\] 
That is, $w_t(0)=\gamma(t)$ and $w_t(1)=\eta(t)$.
Thanks to Lemma~\ref{lemma: 3}, there is $T>0$ such that $\sup\{\|\partial^\circ f(w_t(\lambda))\|:~\lambda\in[0,1]\}<+\infty$ for all $t\geq T$.
Now, thanks to Lemma~\ref{lemma: 2}, we deduce
\[ 
\int_0^1 \big(\|\partial^\circ f(w_t(\lambda))\|+\|\partial^\circ g(w_t(\lambda))\|\big) d\lambda<+\infty,~\text{ for all }t\geq T.
\]

Recalling that $w_t\in \mathcal{W}^{1,2}([0,1],\mathcal{H})$ ($w_t$ is $d(\hat{x},y)$-Lipschitz), 
 we can apply the chain rule for convex functions \cite[Proposition 17.2.5]{ABM2014-book} to obtain 

\begin{align*}
\psi(\hat{x})- \psi(y) \leq \psi (\eta(t))-\psi(\gamma(t))&=\int_0^1 \langle \partial^\circ f(w_t(\lambda)) -\partial ^\circ g(w_t(\lambda)),\dot{w}_t(\lambda)\rangle d\lambda\\
&\leq  d(\hat{x},y)\int_0^1 \big(\|\partial^\circ f(w_t(\lambda)) - p\| +\|\partial^\circ g(w_t(\lambda)) - p\|\big)d\lambda.
\end{align*}

Recall now that, for every $\lambda\in[0,1]$, $\partial^\circ f(w_t(\lambda))$ and $\partial^\circ g(w_t(\lambda))$ converge to $p$ as $t$ goes to $+\infty$.
Since $t\geq T$, we can use Lebesgue Dominated Convergence theorem (using as majorizing function the function 
\(\lambda\in[0,1]\mapsto 2\|\partial^\circ f(w_T(\lambda))\|+2\|p\|\)), 
to deduce that
\begin{align*}
    \psi(\hat{x})-\psi(y)&\leq \limsup_{t\to+\infty} \psi (\eta(t))-\psi(\gamma(t))\leq 0.
\end{align*}
In a similar manner, using the subgradient flow semigroup of $g$, we can show that ${\psi(z)-\psi(\hat{x})\geq 0}$, leading to the equality $\psi(z) = \psi(y)$. It follows that $\psi$ is constant on $\Omega$.\smallskip\newline
Since $f$ is equal to $g$ up to a constant on $\Omega=\mathrm{dom}\,\partial f=\mathrm{dom}\,\partial g$, it readily follows that $f$ is equal to $g$ up to a constant on $\mathcal{H}$ due convexity and lower semicontinuity of both functions.
The proof is now complete. 
\end{proof}

\textbf{Open problems:} It is not known if~\eqref{eq:Conjecture-DeterminationNeumann} holds in full generality. In the smooth case, the main technique is to show that $\|\nabla f - \nabla g\|$ is nonincreasing along the gradient curve $\gamma$. For the case where the Crandall-Pazy direction is attained, the control is given by bounding the distance between the subgradient flows of $f$ and $g$, using the common subgradient curve $\eta(t) = x-tp$, and the contraction property of the subgradient flows. We conjecture that~\eqref{eq:Conjecture-DeterminationNeumann} holds for general lower semicontinuous convex functions. For this case, a new technique is needed to show that subgradient flows of $f$ and $g$ cannot evolute in different directions.\\ 

\textbf{Acknowledgments} A major part of this work has been conducted during research visits of A.~Daniilidis and S.~Tapia-Garc\'{i}a at Universidad de O'Higgins (January 2024, January 2025), and of D.~Salas at TU Wien (November 2024). The first author has been supported by the  FWF grant P36344 (Austria). The second author has been supported by the BASAL grant FB210005 and the FONDECYT grant 1251159 (Chile). The second author acknowledges the support of the CNRS through the \textit{Poste Rouge} program, that financed a 3-month visit to the IMT, Université Toulouse Paul Sabatier, CNRS, INSA Toulouse, UT, INU, UT Capitole, UT2J. \smallskip\newline 
This research was funded in whole or in part by the Austrian Science Fund (FWF) [10.55776/P36344]. For open access purposes, the first author has applied a CC BY public copyright license to any author-accepted manuscript version arising from this submission.

\bibliographystyle{plain}
\bibliography{Unbounded}

	\newpage
	
	\rule{5 cm}{0.5 mm} \bigskip\newline\noindent Aris DANIILIDIS, Sebasti{\'a}n TAPIA-GARC{\'I}A
	
	\medskip
	
	\noindent Institute of Statistics and Mathematical Methods in Economics,
	E105-04 \newline TU Wien, Wiedner Hauptstra{\ss}e 8, A-1040
	Wien\smallskip\newline\noindent E-mail: \{\texttt{aris.daniilidis,
		sebastian.tapia.garcia\}@tuwien.ac.at}\newline\noindent
	\texttt{https://www.arisdaniilidis.at/}\newline
	\texttt{https://sites.google.com/view/sebastian-tapia-garcia}
	
	\medskip
	
	\noindent Research supported by the Austrian Science Fund grant \textsc{FWF
		P-36344N}.\newline\vspace{0.2cm}
	
	\noindent David SALAS
	
	\medskip
	
	\noindent Instituto de Ciencias de la Ingenier\'{i}a, Universidad de
	O'Higgins\newline Av. Libertador Bernardo O'Higgins 611, Rancagua, Chile
	\smallskip
	
	\noindent E-mail: \texttt{david.salas@uoh.cl} \newline\noindent
	\texttt{http://davidsalasvidela.cl} \medskip
	
	\noindent Research supported by the grant: \smallskip\newline CMM 
	FB210005 BASAL, FONDECYT 1251159 (Chile)
	
\end{document}